\documentclass{amsart}
        \usepackage{amsmath}
        \usepackage{amsfonts}
        \usepackage{amsthm}
        \usepackage{subfig}
        \usepackage{url}
        \usepackage{color}
        \usepackage{mathtools}
		\usepackage[shortalphabetic]{amsrefs}
        \usepackage{latexsym}
        \usepackage{amssymb}
		\usepackage{graphicx}        
		\usepackage{float}
        
        \usepackage[colorinlistoftodos]{todonotes}
		\usepackage[colorlinks=true, allcolors=blue]{hyperref}

        \usepackage{eucal}

        \theoremstyle{plain}
        \newtheorem*{theorem*}{Theorem}
        \newtheorem{theorem}{Theorem}[section]
        
        \newtheorem{lemma}[theorem]{Lemma}

        \theoremstyle{definition}
        
        \newtheorem{example}[theorem]{Example}
        
        \newtheorem*{example*}{Example}

        \theoremstyle{remark}
        \newtheorem{remark}[theorem]{Remark}
        \newtheorem*{remark*}{Remark}

\newcommand{\PP}{\mathbb{P}}
\newcommand{\RR}{\mathbb{R}}

\begin{document}
\title{Minimum cross-entropy distributions on Wasserstein balls and their applications.}
\author{Luis Felipe Vargas}
\address{Luis Felipe Vargas, Centrum Wiskunde \& Informatica (CWI)\\  Science Park 123\\ 1098 XG Amsterdam\\ The Netherlands.}
\email{luis.vargas@cwi.nl}
\author{Mauricio Velasco}
\address{Mauricio Velasco, Departamento de
  Matem\'aticas\\ Universidad de los Andes\\ Carrera 1 No. 18a 10\\ Edificio
  H\\ Primer Piso\\ 111711 Bogot\'a\\ Colombia}
\email{mvelasco@uniandes.edu.co}


\begin{abstract}
Given a prior probability density $p$ on a compact set $K$ we characterize the probability distribution $q_{\delta}^*$ on $K$ contained in a Wasserstein ball $B_{\delta}(\mu)$ centered in a given discrete measure $\mu$ for which the relative-entropy $H(q,p)$ achieves its minimum. This characterization gives us an algorithm for computing such distributions efficiently.
\end{abstract}

\maketitle

\section{Introduction}

Kullback's minimum cross-entropy principle is one of the basic mechanisms available for statistical inference. This principle states that among all probability distributions $q$ satisfying a given collection of moment inequalities the ``best'' approximation to a given prior distribution $p$ is the unique $q$ for which the relative entropy $H(q,p)$ achieves its minimum. In many contexts it is known that solving these optimization problems leads to the {\it unique} self-consistent inference mechanism (see Section~\ref{Sec: Stats} for details) which makes it a rather natural approach.

In this article we apply this principle to the basic problem of machine learning. More precisely we assume that we are given an i.i.d. sample $X_1,\dots, X_N$ of a random variable taking values in a set $K\subseteq \RR^n$ and a prior distribution $p$, which represents our beliefs about the distribution of the random variable which generated the data. Our main objective is to learn the distribution of the data from the samples while simultaneously incorporating the information contained in our prior $p$. 

If the number $N$ of samples is sufficiently large then statistical learning theory guarantees that the samples alone suffice to obtain a good approximation of the underlying distribution eliminating the need for a prior distribution $p$. In contrast, in this article we will be interested in situations where the sample size $N$ is assumed to be not too large. In this regime the following two basic questions become fundamental:
\begin{enumerate}
\item How to incorporate the prior information $p$ into our learning mechanism?
\item How to make our inference mechanism more robust to prevent the  possibility of overfitting inherent in the small sample size? 
\end{enumerate}

As we will show, both of these questions have a natural common answer, which will depend on an auxiliary positive real number $\delta$. As the parameter $\delta$ changes the distribution $q^*(\delta)$ learned by our algorithm will change. It will coincide with the empirical distribution of the data when $\delta=0$ and will coincide with the prior distribution $p$ when $\delta=\infty$.  For other values of $\delta$ our inference procedure will interpolate between these two extreme cases, incorporating information from both the prior and the data sample in a consistent manner. The output of our algorithm will be the curve of probability densities $q^*(\delta)$. The parameter $\delta$ is free, to be selected by the user using additional external information. This extra degree of freedom will be very useful for the application we consider in Section~\ref{Sec: Applications}.

More concisely, we aim to solve the problems
\[
q^*(\delta):={\rm argmin}\left\{ H(q,p): q\in B_{\delta}(\mu).\right\}
\]
where $B_{\delta}(\mu)$ is an ambiguity set, typically a ball of radius $\delta$, around the empirical measure $\mu:=\frac{1}{N}\sum_{i=1}^N \delta_{X_i}$ defined by the sample. For the problem to be completely specified we need to select the metric used to define this ambiguity ball. For the chosen metric we should have:
\begin{enumerate}
\item  Quantitative consistency bounds guaranteeing that, with high probability, the true distribution of the $X_i$ lies in $B_{\delta}(\mu)$ for a known radius $\delta=\delta(N)$ which decreases to zero as the number of samples goes to infinity.
\item  An effectively computable minimization problem. Note that this is a nontrivial requirement since the space of distributions is infinite-dimensional unless $K$ is a finite set.
\end{enumerate}
A considerable amount of recent work (see for instance~\cite{quantization},\cite{EsfahaniKuhn}, \cite{Kuhnetal}) has shown that the Wasserstein (or earth-mover's) distance satisfies both of these requirements (see Section~\ref{Sec: Stats}).

To give a more precise description of the Wasserstein metric and of the proposed algorithm we need to introduce some terminology. Assume $K\subseteq \RR^n$ is a regular compact set (i.e. the closure of an open bounded subset of $\RR^n$) with a fixed metric $d$. For probability distributions $\mu,\nu$ on $K$ let $\Pi(\mu,\nu)$ be the set of probability distributions in $K\times K$ whose marginals coincide with $\mu$ and $\nu$ respectively. We can think of $\pi\in \Pi(\mu,\nu)$ as a transportation plan between $\mu$ and $\nu$ and define the Wasserstein distance as the minimum amount of work needed to transform $\mu$ to $\nu$, that is as the number
\[\mathcal{W}(\nu, \eta) : =\inf_{\pi\in \Pi(\mu,\nu)} \int_{K\times K} d(x,y)d\pi(x,y).\]

With these preliminaries we can formally state our approach. Given a real number $\delta>0$, an independent sample $x_1,\dots, x_N$ of a random variable taking values in $K$ and a prior distribution $p$ which is strictly positive on $K$, our proposed learning strategy consists of finding a minimizer $q^*(\delta)$ of the cross-entropy $H(q,p)$  
\begin{equation}
\label{BasicProblem}
q^*(\delta):={\rm argmin}\left\{ H(q,p): q\in B_{\delta}(\mu).\right\}
\end{equation}
where $B_{\delta}(\mu)$ is the set of probability distributions $\nu$ with $\mathcal{W}(\nu,\mu)\leq \delta$. Our first result shows that membership in such Wasserstein balls can be recast as a collection of moment inequalities. This fact justifies cross-entropy minimization as a (in fact the only) self-consistent inference procedure in this context (see Section~\ref{Sec: Stats} for details). To state it, given $\lambda\in \RR^N$ define the function $\phi_{\lambda}(x):=\min_{1\leq i\leq N}\{d(x,x_i)-\lambda_i\}$. 

\begin{theorem}\label{thm: justification} If $\mu=\frac{1}{N}\sum_{i=1}^N \delta_{x_i}$ then a probability density $q(x)\in L^1(dx)$ belongs to $B_{\delta}(\mu)$ if and only if $q(x)$ satisfies the moment constraints
\[\int_K\phi_{\lambda}(x)q(x)dx \leq \delta\text{ for all $\lambda\in \RR^n$ with $\sum_{j=1}^N \lambda_j=0$}.\]
As a result, solving problem~(\ref{BasicProblem}) is the only self-consistent inference method for choosing a posterior density $q\in B_{\delta}(\mu)$ for the given prior $p$.
\end{theorem}

Problem~(\ref{BasicProblem}) is infinite-dimensional and thus does not lend itself to computation immediately. Extending the seminal work by Carlsson, Behroozi and Mihic~\cite{Carlssonetal} we show that it is nevertheless possible to reformulate it so as to make it concave and finite-dimensional. More precisely we prove the following characterization of optimal solutions

\begin{theorem} \label{thm: main} The unique density of minimum cross-entropy in $B_{\delta}(\mu)$ is given by the formula
\[q^*(x):=p(x)\exp\left[-1-\overline{v}\phi_{\lambda^*}(x) -\overline{u}\right].\] 
for a unique $\lambda^*\in \RR^N$ and $(\overline{u},\overline{v})\in \RR^2$. Moreover, given $\lambda^*$ the pair $(\overline{u},\overline{v})$ can be characterized as the unique maximizer of the strictly concave two-dimensional maximization problem
\[\max_{(u,v)\in \RR^2, v\geq 0}\left(-u-v\delta-\int_Kp(x)e^{-\left(1+v\phi_{\lambda^*}(x)+u\right)}dx\right).\]
\end{theorem}

From Theorem~\ref{thm: main} it follows that if we knew the ``magical" value $\lambda^*\in \RR^N$ then we could easily find the desired minimum cross-entropy solution $q^*(\delta)$.

Section~\ref{sec: cutting-plane} is therefore devoted to the problem of finding $\lambda^*$: we characterize $\lambda^*$ as the unique maximizer of a quasi-concave function $\Gamma(\lambda)$ and provide a cutting plane algorithm allowing us to approximate $\lambda^*$ to any desired accuracy. To give a precise description of this characterization we need to introduce some notation. Let $\Lambda:=\{\lambda\in \RR^N: \sum \lambda_i=0\}$ and  
for $\lambda\in \Lambda$ define
\[\Gamma(\lambda) := \min\{H(q,p): q\in \Omega(\lambda)\}\]
where $\Omega(\lambda)$ is the set of probability densities $q$ on $K$ which satisfy the inequality $\int_K \phi_{\lambda}(x) q(x)dx\leq \delta$.
For $\lambda\in \Lambda$ and $i=1,\dots, N$ define the weighted Voronoi region around the point $x_i$ as  
\[R_i(\lambda):=\{x\in K: \forall j\left(d(x,x_i)-\lambda_i\leq  d(x,x_j)-\lambda_j\right)\}.\]
With these notational preliminaries we can characterize $\lambda^*$:

\begin{theorem}\label{thm: main2} The following statements hold:
\begin{enumerate}
\item For any $\lambda\in \Lambda$ the strong Lagrange dual of $\Gamma(\lambda)$ is equivalent to the concave two-dimensional maximization problem

\[\max_{(u,v)\in \RR^2, v\geq 0}\left(-u-v\delta-\int_Kp(x)e^{-\left(1+v\phi_{\lambda}(x)+u\right)}dx\right).
\]

\item If $(\overline{u},\overline{v})$ are maximizers of the problem in part $(1)$ and 
\[\overline{q}_{\lambda}(x):=p(x)\exp\left[-1-\overline{v}\phi_{\lambda}(x) -\overline{u}\right]\] 
then the vector $g\in \RR^N$ given by $g_i:=\frac{1}{N}-\int_{R_i(\lambda)} \overline{q}_{\lambda}(x)dx$ defines a halfspace 
\[\{\lambda'\in \Lambda: \langle g,\lambda'-\lambda\rangle\geq 0\}.\] 
which contains the maximizers of $\Gamma$.
\item The function $\Gamma(\lambda)$ is quasi-concave and has the value $\lambda^*$ from Theorem~\ref{thm: main} as its unique maximizer. 

\end{enumerate}
\end{theorem}

Theorem~\ref{thm: main2} part $(3)$ explains how to construct separators for $\lambda^*$ and thus can be used to construct a sequence of polyhedra of diminishing volume converging to $\lambda^*$ (see Section~\ref{sec: cutting-plane} for details).  

In Section~\ref{Sec: Applications} we discuss several computational experiments carried out with our algorithm. Finally, in Section~\ref{Sec: Applications} we also suggest a practical application of these ideas. We propose computing minimum entropy distributions as a mechanism to mitigate bias in machine learning algorithms (as defined by social scientists). 

{\bf Acknowledgments.}  M Velasco was partially supported by ECOSNord Colciencias grant Problemas de momentos en control y optimizaci\'on (C\'odigo 62910, Convocatoria: 806-2018) and 
by proyecto INV-2018-50-1392 from Facultad de Ciencias, Universidad de los Andes. L.F Vargas is partially supported by the European Union's Framework Programme for Research and Innovation Horizon 2020 under the Marie Sklodowska-Curie Actions Grant Agreement No. 813211  (POEMA). We  wish to thank Fabrice Gamboa, Adolfo Quiroz and \'Alvaro Riascos for many stimulating discussions during the completion of this work.

\section{Preliminaries on cross-entropy and statistical inference on Wasserstein balls}
\label{Sec: Stats}
Kullback's principle of minimum cross-entropy~\cite{Kullback} gives a general method of inference about an unknown probability density $q$ when we are given a prior estimate $p$ of $q$ and inequality constraints on the expected values (moments) of a collection of functions under the unknown distribution $q$. The principle states that one should choose $q$ to be the probability density  satisfying the moment constraints for which the cross entropy $H(q,p)$ is minimized. 

Recall that the cross-entropy (also known as Kullback-Liebler divergence, $I$-divergence or information gain from $p$ to $q$) of two probability densities $q$ and $p$ on $K$ is defined as
\[H(q,p):=\int_K \log\left( \frac{q(x)}{p(x)} \right)q(x)dx.\]

There are many justifications for this principle of which we would like to emphasize two. The first one relies on the properties of cross-entropy as an {\it information measure}. In the discrete case, by~\cite{Hobson} the nonnegative number $H(q,p)$ can be interpreted as the smallest amount of information necessary to change the prior $p$ to the posterior $q$. It is shown in~\cite{Johnson} that similar axiomatic properties are satisfied by cross-entropy in the continuous case. Solving problem~(\ref{BasicProblem}) can therefore be thought of as choosing, among all distributions in $B_{\delta}(\mu)$, the one which can be obtained from $p$ by using the smallest possible amount of additional information. 

The second more formal justification comes from the fundamental work of Shore and Johnson~\cite{SJ1},\cite{SJ2} who show that cross-entropy minimization is {\it the only} self-consistent inference method for, given a prior $p$, selecting a posterior density $q$ from a set $I$ of densities on $K$ satisfying a collection of {\it moment inequality constraints}.  More precisely,  if we write $q=p\circ I$ to denote any such selection procedure then one would expect that any self-consistent method satisfies the following four axioms: 

\begin{enumerate}
\item {\it Uniqueness:} There is a unique solution $q$ for any $I$ and $p$ so $q=p\circ I$ is well defined.
\item {\it Invariance:} The chosen distribution is the same, regardless of the choice of coordinates in which we solve the problem (i.e. if $\Gamma$ is a diffeomorphism of the domain then $\Gamma^*p\circ \Gamma^* I = \Gamma^*( p\circ I)$). 
\item {\it System independence:} Given priors $p_1,p_2$ and constraint sets $I_1,I_2$ about the two systems then the following equality holds:
\[ p_1p_2\circ (I_1\cap I_2)  = (p_1\circ I_1)(p_2\circ I_2).\]
That is it should not matter whether one accounts for independent information about independent systems separately in terms of different densities or together in terms of a joint density.
\item {\it Subset independence:} If $S\subseteq K$ and $q$ is a density on $K$ then denote by $q\ast S$ the conditional density $q(x|x\in S)$.
If $S_1,\dots S_k$ are a partition of $K$, the sets $I_i$ are constrains on moments of the density of $q$ {\it conditioned to $S_i$} and $I=\bigcap_{i=1}^n I_i$ then one would expect that 
\[(p\circ I)\ast S_i = (p\ast S_i)\circ I_i\] 
In words it should not matter whether one treats an independent subset of system states in terms of a separate conditional density or in terms of the full system density.  
\end{enumerate}
The main result of~\cite{SJ1}[Theorem III] is that the principle of minimum cross-entropy is {\it the only inference procedure} satisfying the four axioms above. We are now in a position to prove Theorem~\ref{thm: justification}.

\begin{proof}[Proof of Theorem~\ref{thm: justification}]
This is an immediate consequence of Lemma~\ref{lem: supergradient} part $(1)$ in the following Section and~\cite{SJ1}[Theorem III]. 
\end{proof}

Another motivation for using the Wasserstein distance to define  our ambiguity sets is the fact that there are well-known estimates of the distance between the true distribution $\hat{p}$ and the empirical measure $\mu$ determined by the sample $X_i$ which give statistical consistency guarantees to our approach. More concretely, by~\cite[Theorem 1]{quantization} we know that there exists a constant $\kappa$ depending on $K$ and $\hat{p}$, such that 
\[\mathbb{E}\left[\mathcal{W}(\hat{p},\mu)\right]\leq \kappa N^{-\frac{1}{n}}\] 
from which Markov's inequality implies that 
\[\PP\{\mathcal{W}(\hat{p},\mu)\geq \delta(N)\}\leq \frac{\kappa N^{-\frac{1}{n}}}{\delta(N)}.\]
As a result, if for any $\epsilon>0$ we set $\delta(N)= \frac{\kappa N^{-\frac{1}{n}}}{\epsilon}$ then the true distribution is guaranteed to lie in $B_{\delta(N)}(\mu)$ with probability at least $1-\epsilon$ and $\delta(N)\rightarrow 0$ as $N\rightarrow \infty$.  As a result, the additional information $H(\hat{p},p)$ needed to transform the prior $p$ into the true distribution $\hat{p}$ satisfies the inequality 
\[H(\hat{p},p)\geq H(q^*,p).\]
where $q^*:=q^*(\delta(N))$ is the minimum entropy distribution in $B_{\delta(N)}(\mu)$. Moreover, it can be shown that the right-hand side converges to the true value as $N\rightarrow \infty$ with high probability. The results in this paper allow us to compute the quantity on the right-hand side effectively and thus to {\it learn} an estimate of the information contained in a process from an i.i.d. sample of it. 

\section{Computing optimal transports to discrete measures}

Throughout the rest of the article we assume that $K\subseteq \RR^n$ is a regular compact set (i.e. the closure of a bounded open subset of $\RR^n$) with a fixed metric $d$. We endow $K$ with its Lebesgue measure (denoted by $dx$). We assume that the metric is continuous and that for every $t\in \RR$ and $x_i,x_j\in K$ the following set has measure zero: 
\[\{x\in K: d(x,x_i)-d(x,x_j)=t\}.\]

Moreover, we assume that $\mu=\frac{1}{N}\sum_{i=1}^N \delta_{x_i}$ is the empirical measure of our given sample $x_i\in K$ and that $p$ is a continuous and strictly positive probability density function in $K$. By a density we mean a function $q\in L^1(dx)$ with $q(x)\geq 0$ and $\int_Kq(x)dx=1$ such a density specifies, via integration, a corresponding probability measure (distribution). We will use the letter $q$ to denote both the density and the corresponding distribution and let $\Lambda=\{\lambda\in \RR^N: \sum\lambda_i=0\}$.

For a fixed probability density $q$ on $K$ define, for $\lambda\in \RR^N,x\in K$ the functions 
\[
\begin{array}{ll}
\phi_{\lambda}(x):=\min_{1\leq i\leq N} \left\{d(x,x_i)-\lambda_i\right\} & \Psi(\lambda) = \int_K \phi_{\lambda}(x)q(x)dx.\\
\end{array}
\]

Our first Lemma allows us to construct an optimal transport from $q$ to $\mu$ by solving a finite-dimensional concave maximization problem. It summarizes ideas contained in~\cite{Carlssonetal} whose proof we include for the reader's benefit. The proof also explains how the function $\phi_{\lambda}$ arises naturally in this context.
\begin{lemma}\label{lem: supergradient} The following statements hold:
\begin{enumerate}
\item The Wasserstein distance $\mathcal{W}(q,\mu)$ can be computed as
\[
\mathcal W(q,\mu) = \max\left\{ \Psi(\lambda)\text{ for $\lambda\in \RR^N$ with $\sum \lambda_i=0$ }\right\}.\]
\item For $i=1,\dots, N$ define the regions 
\[R_i(\lambda):=\left\{x\in \RR^n: \forall j\left(d(x,x_i)-\lambda_i \leq d(x,x_j)-\lambda_j\right)\right\}.\]
The function $\Psi(\lambda)$ is concave and moreover
\[\nabla\Psi(\lambda) = \left(\frac{1}{N}-\int_{R_i(\lambda)} q(x)dx\right)_{i=1,\dots, N}.\]
\item An optimal transport between $q$ and $\mu$ is obtained by sending all points in the region $R_i(\lambda^*)$ to the point $x_i$ for any maximizer $\lambda^*$ of $\Psi$ in $\Lambda$.
\item If $q(x)>0$ almost surely in $K$ then there is a unique maximizer $\lambda^*$ of $\Psi$ in $\Lambda$.

\end{enumerate}

\end{lemma}
\begin{proof}$(1)$ The Kantorovich duality Theorem~\cite[Section IV,14]{Barvinok} asserts that the equality
\[\inf_{\pi\in \Pi(\mu,\nu), (x,y)\sim \pi} \int_{K\times K} d(X,Y)d\pi(x,y) =\max \int_K a(x)d\mu(x) +\int_K b(y)d\nu(y) \]
holds, where the maximum on the right is taken over all pairs of real valued functions $a,b$ on $K$ for which $a(x)+b(y)\leq d(x,y)$ almost surely with respect to any $\pi$. If the measure $\mu$ is discrete and supported on $x_1,\dots, x_N$ then:
\begin{enumerate}
\item The only values of $a(x)$ that matter for the final integral are the numbers $a(x_i)=:\lambda_i$ and the function $a(x)$ enters the objective only through the sum $a(x)+b(y)$. As a result we can add a constant to $a(x)$ and substract it from $b(y)$ without changing the objective function.
\item The inequality $a(x)+b(y)\leq d(x,y)$ becomes $b(y)\leq d(x_i,y)-a(x_i)$ for all $y\in K$ and $i=1,\dots, N$. Given the $\lambda_i$ the best (largest) choice for $b$ is therefore 
\[b(y)=\min_{1\leq i\leq n}\left\{d(x_i,y)-\lambda_i\right\}=\phi_{\lambda}(y).\]
\end{enumerate}
It follows that \[\mathcal{W}(\mu,\nu)= \sup_{\lambda\in \RR^N}\left(\int_K \phi_{\lambda}(y)d\nu(y)+ \sum_{i=1}^N \frac{\lambda_i}{N} \right)= \sup_{\lambda\in \Lambda}\int_K \phi_{\lambda}(y)d\nu(y) =\sup_{\lambda\in \Lambda} \Psi(\lambda) \]
where $\Lambda=\{\lambda\in \RR^N: \sum\lambda_i=0\}$ proving part $(1)$.
$(2)$ The function $\Psi(\lambda)$ is an average of minima of affine linear functions of $\lambda$ and is therefore concave. Differentiating inside the integral sign and using our assumptions on the metric we obtain the vector 
\[g=\left(-\int_{R_i(\lambda)} q(x)dx\right)_{i=1,\dots, N}\] whose orthogonal projection onto the subspace $\Lambda$ is given by the claimed formula proving $(2)$. In particular for any maximizer $\lambda^*$ of $\Psi$ the equality 
\[\int_{R_i(\lambda^*)} q(x)dx=\frac{1}{N}\]
holds for $i=1,\dots, N$.
$(3)$ By the previous equality, the function which sends  every point $x\in R_i(\lambda^*)$ to $x_i$ for $i=1,\dots, N$ defines a transportation plan $\pi$ with cost
\[\int_{K\times K} d(x,y)d\pi(x,y) = \sum_{i=1}^N  \int_{R_i(\lambda^*)} d(x_i,x)q(x)dx =\]
\[=  \sum_{i=1}^N \int_{R_i(\lambda^*)} \left(d(x_i,x)-\lambda^*_i\right)q(x)dx+ \sum_{i=1}^N \frac{\lambda_i^*}{N} =\]
\[=\int_K \phi_{\lambda^*}(x)q(x)dx =\Psi(\lambda^*)=\mathcal{W}(q,\mu) \]
so $\pi$ is an optimal transportation plan between $q$ and $\mu$ as claimed.

$(4)$ Assume that $q(x)>0$ a.s. and suppose $T_1,\dots, T_N$ is any partition of $K$ into $N$ regions. Note that for every $i=1,\dots, N$ we have
\[\int_{R_i(\lambda)} \left(d(x_i,x)-\lambda_i\right) q(x)dx \leq  \sum_{j=1}^N \int_{R_i(\lambda)\cap T_j} (d(x_j,x)-\lambda_j)q(x)dx \] 
because the regions $R_i(\lambda)$ are defined as the set where the functions on the right-hand side achieve the minimum at index $i$. Crucially the inequality is {\it strict} whenever there is some $j\neq i$ such that $R_i(\lambda)\cap T_j$ has positive measure because $q(x)>0$. 
Summing these inequalities over $i$ for $T_j:=R_j(\lambda^*)$ for $j=1,\dots, N$, we conclude that 
\[\Psi(\lambda)= \sum_{i=1}^N\int_{R_i(\lambda)} \left(d(x_i,x)-\lambda_i\right) q(x)dx<\sum_{i=1}^N\sum_{j=1}^N \int_{R_i(\lambda)\cap R_j(\lambda^*)} (d(x,x_j)-\lambda_j)q(x)dx =\]
\[= \sum_{j=1}^N \int_{R_j(\lambda^*)} (d(x,x_j)-\lambda_j)q(x)dx = \int_{K}d(x,x_j)q(x)dx =\Psi(\lambda^*)\]
where the equalities follow from the fact that $\lambda\in \Lambda$ and that $\int_{R_j(\lambda^*)}q(x)dz=\frac{1}{N}$ for every index $j$ as shown in the proof of part $(2)$. 
\end{proof}

Part $(2)$ of the previous Lemma allows us to construct the optimal transport between $q$ and $\mu$ whenever we are able to compute (or at least approximate) the integrals $\int_{R_i(\lambda)} q(x)dx$ of $q$ over weighted Voronoi regions. This computation can be carried out using Montecarlo methods whenever one can efficiently sample a random variable with density $q(x)$. More precisely, choose $\lambda^{(0)}\in \Lambda$ and define iterates
\[\lambda^{(k)}:= \lambda^{(k-1)} +t_k g^{(k-1)}\]
where $g^{(k-1)}$ is the gradient of $\Psi$ at $\lambda^{(k-1)}$ computed in Lemma~\ref{lem: supergradient} part $(2)$ and $t_{k}\geq 0$ is a step-size By the concavity of $\Psi$ the sequence $\lambda^{(k)}$ converges to the maximizer $\lambda^{*}$ of $\Psi$ whenever the sequence $(t_k)_k$ of step sizes is chosen to be square summable but not summable. 

\begin{remark} As in the introduction, the regions $R_i(\lambda)$ are called a {\it Voronoi diagram additively weighted by $\lambda$}. The reason for this nomenclature is that in the Euclidean case (i.e. when $d(x,y)=\|x-y\|_2$) and $\lambda=0$ the corresponding regions $R_i(\lambda)$ are the usual Voronoi cells determined by the points $x_i$. Like Voronoi diagrams, weighted Voronoi diagrams are very beautiful combinatorial structures (see Figure~\ref{Fig: Voronoi}.) 
\end{remark}

\begin{figure}[h]
    \centering
    \subfloat[][\centering Weighted Voronoi Diagram]{\includegraphics[width=5.8cm]{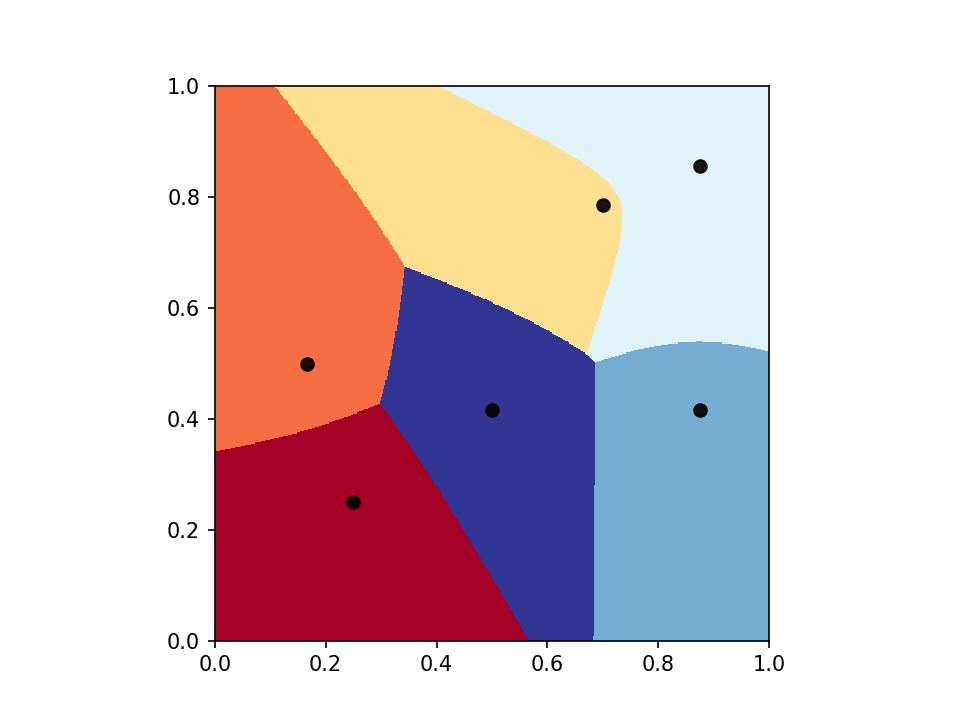} 
    }
    \qquad
    \subfloat[][\centering (Unweighted) Voronoi diagram]{\includegraphics[width=5.8cm]{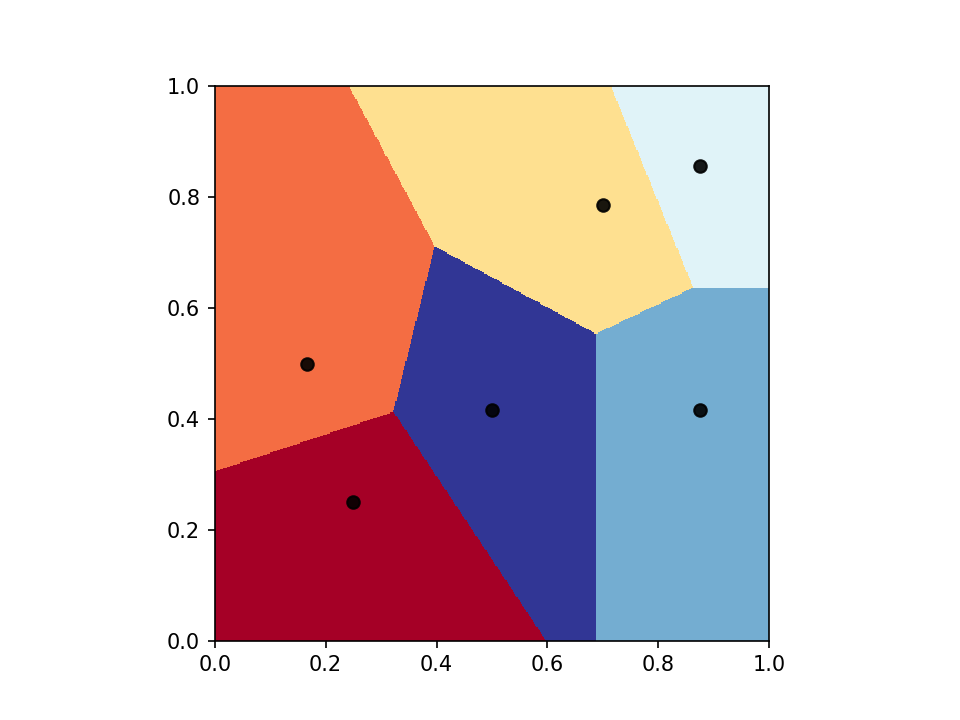} }
    \caption{Figure $(A)$ describes an optimal transport between the uniform distribution on $[0,1]\times[0,1]$ and the given empirical measure by sending each color class to the unique black dot in its interior. It was computed via Lemma~\ref{lem: supergradient}.}
    \label{Fig: Voronoi}%
\end{figure}

\section{Computing minimum cross-entropy distributions in Wasserstein balls centered at discrete distributions.}
In this Section we develop an algorithm for finding a minimizer of the optimization problem: 
\begin{equation}
\label{BasicProblem2}
\min_q\left\{ H(q,p): \mathcal{W}(q,\mu)\leq \delta \right\}
\end{equation}
In view of Lemma~\ref{lem: supergradient} part $(1)$ this problem can be reformulated as

\begin{equation}
\label{BasicProblem2v2}
\min_q\left\{ H(q,p): \int_K \phi_{\lambda}(x) q(x)dx \leq \delta \text{ for every $\lambda\in \RR^N$ with $\sum \lambda_i=0$} \right\}
\end{equation}
where $q$ runs over the nonnegative densities $q\in L^1(dx)$ which integrate to one.
A key step for solving this problem will be the study of its Lagrangian dual. We begin by reminding the reader of a version of Lagrange duality suitable for our application to infinite-dimensional linear spaces.

\subsection{Lagrange duality} \label{Sec: LagrangianDual} Let $V$ be a vector space and let $Z$ be a normed vector space. We denote the continuous dual of $Z$ via $Z^*$ and denote the usual pairing between these spaces with $\langle z,z^*\rangle$. We endow $Z$ with a closed cone $P$ of positive elements and $Z^*$ with a corresponding dual cone $P^*:=\{z^*\in Z^*: \forall z\in P\left(\langle z,z^*\rangle \geq 0\right)\}$. The positive cone $P$ allows us to define a partial order on $Z$ via $z_1\preceq z_2$ whenever $z_2-z_1\in P$. Recall that a function $G:V\rightarrow Z$ is called convex if the following inequality holds in the partial order defined by $P$
\[\forall v,v'\in V, \lambda\in [0,1]\left(G(\lambda v+(1-\lambda)v')\preceq \lambda G(v)+ (1-\lambda)G(v')\right).\] 

We fix a convex function $G: V\rightarrow Z$, a convex set $\Omega\subseteq V$, a finite-dimensional vector space $Y$ and an affine linear map $H:V\rightarrow Y$. The following Theorem is known as (strong) Lagrange duality:

\begin{theorem}\cite[Theorem 1 and Problem 7, Section 8.6]{Luenberger} Let $\phi:\Omega\rightarrow \RR$ be a convex real-valued function and define  \[\alpha:=\inf\{\phi(v):v\in \Omega, G(v)\preceq 0, H(v)=0\}\]
If $0$ is an interior point of $H(\Omega)$, $\alpha$ is finite and there exists $v\in V$ satisfying $G(v)\prec 0$ and $H(v)=0$ then
\[\alpha = \sup_{z^*\in P^*,w^*\in Y^*}\left(\inf_{v\in \Omega} \phi(v) + \langle G(v),z^*\rangle+\langle H(v),w^*\rangle\right)\]
and the supremum in the right-hand side is achieved by some $z_0^*\in P^*, w_0^*\in Y^*$. If moreover the left-hand side is achieved by some $v_0\in \Omega$ then:
\begin{enumerate}
\item The equalities $\langle G(v_0), z_0^*\rangle=0$ and $\langle H(v_0), w_0^*\rangle=0$ hold and
\item The point $v_0$ is a minimizer of $\inf_{v\in \Omega}\left(\phi(v) + \langle G(v),z_0^*\rangle+\langle H(v),w_0^*\rangle\right)$.
\end{enumerate}
\end{theorem}

The following example shows that Problem~(\ref{BasicProblem2}) above can be naturally formulated in this setting. As we will show in the next Section the assumptions for Lagrange duality are satisfied, allowing us to rewrite our original problem in a manner amenable to computation.

\begin{example}\label{exm} Let $V=L^1(dx)$ be the space of Lebesgue integrable real valued functions on $K$ and let $\Omega\subseteq V$ be the closed convex cone of a.e. nonnegative functions. Let $\Lambda:=\{\lambda\in \RR^N: \sum \lambda_i=0\}$ and let $Z$ be the space of continuous real-valued functions on $\Lambda$ with the supremum norm endowed with the closed cone $P\subseteq Z$ of functions that are nonnegative on $\Lambda$. By the Riesz representation Theorem the cone $P^*$ is the cone of (unsigned) Borel-measures on $\Lambda$. Let $G:\Omega\rightarrow Z$ be the map which sends $q\in \Omega$ to the continuous function 
\[G(q)(\lambda):=\int_K\phi_{\lambda}(x)q(x)dx-\delta=\Psi(\lambda)-\delta\] and note that $G$ is affine-linear and therefore convex. Let $Y=\RR$ and define the affine-linear function $H:V\rightarrow Y$ via $H(q)=\int_Kqdx-1$.
Define $\phi: \Omega\rightarrow \RR$ via $\phi(q)=\int_Kq(x)\log\frac{q(x)}{p(x)}dx$ and note that $\phi$ is well-defined because $p(x)$ is strictly positive on $K$. Note that $G(q)\preceq 0$ if and only if $\Psi(\lambda)\leq \delta$ for every $\lambda\in \Lambda$. It follows from Lemma~\ref{lem: supergradient} part $(1)$ that problem~(\ref{BasicProblem2}) is equivalent to finding
\[\alpha = \inf_{q\in \Omega,G(q)\preceq 0, H(q)=0} \phi(q)\]
\end{example}

We are now ready to prove the main result of this Section,
\begin{proof}[proof of Theorem~\ref{thm: main}]
We will use the notation from Example~\ref{exm}. We wish to apply Lagrange duality and thus begin by verifying the hypotheses of Theorem~\ref{Sec: LagrangianDual}. For any $q\in L^1(dx)$ the inequality $H(q,p)\geq 0$ holds and therefore $\alpha>-\infty$. If $\hat{q}\in \Omega$ is a probability density function with $W(\hat{q},\mu)=\frac{\delta}{2}$ then $G(\hat{q})\prec 0$ and $H(q)=0$. Moreover there exists $\epsilon>0$ such that $\alpha\hat{q}\in \Omega$ for $|\alpha|\leq \epsilon$ and in particular the set $H(\Omega)$ contains an interval around $0$. By Theorem~\ref{Sec: LagrangianDual} we conclude that strong duality holds, that is:
\[\inf_{q\in\Omega: G(q)\preceq 0,H(q)=0} H(q,p) = \sup_{z^*\in P^*,u^*\in \RR}\left(\inf_{q\in \Omega} H(q,p) + \langle G(q),z^*\rangle+\langle H(q),u^*\rangle\right).\]

Moreover, the functions $G$ and $H$ are continuous on $V$ and therefore 
\[\left\{q\in \Omega: G(q)\preceq 0,H(q)=0\right\}=\Omega\cap G^{-1}(P)\cap H^{-1}(\{0\})\]
is a closed set in $L^1(dx)$. By~\cite[Theorem 2.1]{Csiszar} it follows that there exists a minimizer $q^*$ of the cross-entropy on this set and moreover that this minimizer is unique because the set is convex and $H(q,p)$ is strictly convex in $q$. Furthermore, since $K$ is a regular compact set, we know that for every open set $A\subseteq K$ there exists a Lebesgue density in $B_{\delta}(\mu)$ which assigns positive measure to $A$. We conclude by~\cite[Remark 2.14]{Csiszar} that the locus of points where $q^*(x)=0$ has Lebesgue measure zero. 

The existence of the a.s. positive minimizer $q^*$ allows us to apply the second part of Theorem~\ref{Sec: LagrangianDual} and obtain the following conclusions:
\begin{enumerate}
\item If $z^*\in P^*$ and $u^*\in \RR^*$ are maximizers of the dual problem above then:
\[\langle G(q^*),z^*\rangle = \int_{\Lambda} G(q^*)(\lambda)dz^*(\lambda)=0\] 
Since $G(q^*)(\lambda)$ is a continuous function and $q^*(x)>0$ almost surely,  Lemma~\ref{lem: supergradient} part $(4)$ implies that the function $G(q^*)$ has value zero at a unique $\lambda^*\in \RR^N$. We conclude that the measure $z^*$ must be a nonnegative real multiple of a Dirac delta measure centered at $\lambda^*$ (i.e. that $z^*=\overline{v}\delta_{\lambda^*}$  for some real number $\overline{v}\geq 0$). 

\item The optimum $q^*$ is therefore a minimizer of the problem
\[\inf_{q\in \Omega} \left(H(q,p) + \overline{v}G(q)(\lambda^*) + \overline{u}\left(\int_Kq(x)dx-1\right)\right)\]
for some $(\overline{u},\overline{v},\lambda^*)\in \RR^2\times \RR^N$ with $\overline{v}\geq 0$. This optimization problem can be written more explicitly as
\begin{equation}
\label{simple}
-\overline{u}-\overline{v}\delta + \inf_{q\in \Omega}\int_K\left(\log q(x)+\overline{v}\phi_{\lambda^*}(x)+\overline{u}-\log p(x)\right)q(x)dx.
\end{equation}
\end{enumerate}

Crucially, the infimum in~(\ref{simple}) can be solved analytically (for any given $(\overline{u},\overline{v},\lambda^*)$) because, for each fixed value of $x$ the problem of choosing $q(x)$ to minimize the integrand reduces to that of minimizing $z$ in $h(z)=z\log(z)+Az$ where $A=\overline{v}\phi_{\lambda^*}(x)+\overline{u}-\log p(x)$. Since $h$ is strictly convex on $z>0$ its unique minimum is achieved whenever $h'(z)=0$ or equivalently when $z=e^{-(1+A)}$. We conclude that the unique pointwise minimum of the integrand above is given by chosing
\[q^*(x):=p(x)\exp\left[-1-\overline{v}\phi_{\lambda^*}(x) -\overline{u}\right]\] 
as claimed. To finish the proof we will characterize the pair $(\overline{u},\overline{v})$ given $\lambda^*$. 

To this end define the function $\Phi: Z^*\times \RR\rightarrow \RR$ via 
\[\Phi(z^*,u) = \inf_{\phi\in \Omega}\left(H(q,p) + \langle G(q),z^*\rangle+uH(q)\right)\] 
and note that by taking suprema over successively smaller sets we obtain the inequalities
\[\sup_{z^*\in P^*,u\in \RR}\Phi(z^*,u)\geq \sup_{z^*=v\delta_{\lambda^*},v\geq 0, u\in \RR} \Phi(z^*,u) \geq \Phi\left(\overline{v}\delta_{\lambda^*},\overline{u}\right)\]
which we know are in fact equalities by the previous paragraph. It follows that $(\overline{u},\overline{v})$ are maximizers of the middle problem, which more explicitly can be rewritten as

\[
\max_{(u,v)\in \RR^2,v\geq 0}\left(-u-v\delta + \inf_{q\in \Omega}\int_K\left(\log q(x)+v\phi_{\lambda^*}(x)+u-\log p(x)\right)q(x)dx\right).
\]
Analytically solving the interior infimum as before we conclude that the unique minimizer has the form
\[q(x):=p(x)\exp\left[-1-v\phi_{\lambda^*}(x) -u\right]\] 
and replacing this expression in the objective function we conclude that $(\overline{u},\overline{v})$ is a maximizer of the concave maximization problem

\[\max_{(u,v)\in \RR^2, v\geq 0}\left(-u-v\delta-\int_K p(x)e^{-\left(1+v\phi_{\lambda^*}(x)+u\right)}dx\right).
\]

A simple direct calculation shows that the gradient of its objective function and its Hessian at a point $(u,v)$ are given by the vector

\[\left( -1+\int_K q(x)dx, -\delta+\int_K\phi_{\lambda^*}(x)q(x)dx\right)\]
and by the symmetric matrix 
\[\left(
\begin{array}{cc}
-\int_K q(x)dx & -\int_K \phi_{\lambda^*}(x)q(x)dx\\
-\int_K \phi_{\lambda^*}(x)q(x)dx & -\int_K \phi_{\lambda^*}(x)^2q(x)dx\\ 
\end{array}\right)\]
which is negative definite showing that the problem is strictly concave and that the point $(\overline{u},\overline{v})$ is the unique maximizer in the convex feasible region we are considering.

\end{proof}

\subsection{A cutting plane algorithm for finding minimum cross-entropy distributions.}\label{sec: cutting-plane}

Theorem~\ref{thm: main} allows us to find the minimum entropy distribution in $B_\delta(\mu)$, provided we know the special value $\lambda^*\in \Lambda$. In this Section we first characterize $\lambda^*$ as the maximum value of a quasi-concave optimization problem and then provide a cutting plane algorithm for approximating its value to any desired accuracy.

Recall from the introduction that for $\lambda\in \Lambda$ we define
\[\Gamma(\lambda) = \min\{H(q,p): q\in \Omega(\lambda)\}\]
where $\Omega(\lambda)$ is the set of $q$ of continous functions in $K$ which satisfy:
\begin{enumerate}
\item $\forall x\in K\left(q(x)\geq 0\right)$ and $\int_Kq(x)dx=1$
\item $\int_K \phi_{\lambda}(x) q(x)dx\leq \delta$.
\end{enumerate}
It is immediate that the set $\Omega(\lambda)$ is a superset of the feasible set of Problem~(\ref{BasicProblem2}) so $\Gamma(\lambda)$ is a lower bound for the optimum of Problem~(\ref{BasicProblem2}). We will show that this relaxation is exact and that $\lambda^*$ is the unique maximizer of $\Gamma$. This provides us with a strategy for finding $\lambda^*$, namely the maximization of $\Gamma$. We are now ready to prove the main result of this Section,

\begin{proof}[Proof of Theorem~\ref{thm: main2}] 
$(1)$ Arguing as in the proof of Theorem~\ref{thm: main} one shows that, for any $\lambda\in \Lambda$, the strong Lagrange dual of problem of $\Gamma(\lambda)$ is given by 
\[\max_{(u,v)\in \RR^2, v\geq 0}-u-v\delta-\left(\min_{q(x)\geq 0}\int_K q(x)\log\frac{q(x)}{p(x)}+q(x)(u+v\phi_{\lambda}(x))dx\right).\]
For any $u,v$ we can minimize the integrand pointwise by selecting \[q(x):=p(x)\exp\left[-(1+v\phi_{\overline{\lambda}}(x)+u)\right]. \] 
replacing this expression in the objective function we see that if $(u(\lambda),v(\lambda))$ are chosen to be the unique optima of 
\[\max_{(u,v)\in \RR^2, v\geq 0}\left(-u-v\delta-\int_Kp(x)e^{-\left(1+v\phi_{\lambda}(x)+u\right)}dx\right) \]
then a distribution of minimum cross-entropy in $\Omega(\lambda)$ is given by  
\[\overline{q}_{\lambda}(x):=p(x)\exp\left[-(1+v(\lambda)\phi_{\lambda}(x)+u(\lambda))\right]. \] 
Furthermore this distribution is unique since $H(q,p)$ is strictly convex in $q$ and $\Omega(\lambda)$ is convex. $(2)$ By Lemma~\ref{lem: supergradient} we know that the vector $g$ is the gradient of the concave function $\Psi(\lambda)$ at $\overline{\lambda}$. As a result for every $\lambda'$ with $\langle \overline{g}, \lambda'-\overline{\lambda}\rangle\leq 0$ the inequality
\[\Psi(\lambda')\leq \Psi(\overline{\lambda})+ \langle \overline{g}, \lambda'-\overline{\lambda}\rangle\]
holds. We conclude that for all such $\lambda'$ the inclusion $\Omega(\lambda')\subseteq \Omega(\overline{\lambda})$ holds and therefore $\Gamma(\lambda')\leq \Gamma(\overline{\lambda})$. It follows that the opposite inequality $\langle \overline{g}, \lambda'-\overline{\lambda}\rangle\geq 0$ must hold at all maximizers $\lambda'$ of $\Gamma$ so $g$ defines a cutting plane as claimed. $(3)$ For each $x\in K$ the function $\phi_{\lambda}$ is a minimum of affine linear functions of $\lambda$ and therefore satisfies the concavity inequality
\[ \phi_{\alpha\lambda+ (1-\alpha)\lambda'}(x)\geq \alpha \phi_{\lambda}(x)+(1-\alpha)\phi_{\lambda'}(x)\]
it follows that whenever $q\in \Omega(\alpha\lambda+(1-\alpha)\lambda')$
\[\delta\geq \int_K \phi_{\alpha\lambda+ (1-\alpha)\lambda'}(x)q(x)dx \geq \alpha \int_K\phi_{\lambda}(x)q(x)dx+(1-\alpha)\int_K\phi_{\lambda'}(x)q(x)\]
so at least one of the summands in the right-hand side is bounded above by $\delta$. We conclude that if $q$ is a minimizer of the optimization problem $\Gamma(\alpha\lambda+(1-\alpha)\lambda')$ then it belongs to either $\Omega(\lambda)$ or $\Omega(\lambda')$ making the minimum over one of these sets a possibly smaller quantity. As a result the following quasi-concavity inequality holds 
\[\Gamma(\alpha\lambda+(1-\alpha)\lambda')\geq \min \left(\Gamma(\lambda),\Gamma(\lambda')\right).\]
By part $(1)$ and Theorem~\ref{thm: main} we know that $\lambda^*$ is a maximizer of $\Gamma$ and that $q^*:=\overline{q}_{\lambda^*}$ is the unique distribution achieving the minimum cross-entropy. If $\overline{\lambda}$ is any maximizer of $\Gamma$ then $B_{\delta}(\mu)\subseteq \Omega(\overline{\lambda})$. The uniqueness of the cross-entropy minimizer in a convex set imply that $\overline{q}_{\overline{\lambda}}=q^*$ from which we conclude $\overline{\lambda}=\lambda^*$. It follows that $\lambda^*$ is the unique maximizer of $\Gamma$ as claimed.

\end{proof}

The previous Theorem allows us to propose a cutting-plane algorithm for finding the distribution $q^*(\delta)$ of minimum cross-entropy in $B_{\delta}(\mu)$. To do this begin with a polytope $P_0\subseteq \RR^N$ which is guaranteed to contain the maximizers of $\Gamma$ (for instance $\|\lambda\|_{\infty} \leq {\rm diam}(K)$) and at each stage $k\geq 0$ repeat the following steps:
\begin{enumerate}
\item Find the Chebyshev center $\lambda^{(k+1)}$ of $P_k$. This is the center of the largest euclidean ball contained in $P_k$ and can be found efficiently by solving a linear optimization problem as in~\cite[Section 8.5.1]{BV}.
\item Solve the problem appearing in Theorem~\ref{thm: main2} part $(1)$ with $\lambda=\lambda^{(k+1)}$ finding a maximizer $(\overline{u}_{k+1},\overline{v}_{k+1})$ (for instance via gradient descent). Use $(\overline{u}_{k+1},\overline{v}_{k+1})$ to define a density $\overline{q}^{k+1}$ with the formula appearing in Theorem~\ref{thm: main2} part $(2)$. This density is a minimizer of the problem $\Gamma(\lambda^{(k+1)})$.

\item Compute the vector $\overline{g}^{(k+1)}$ as
\[\overline{g}^{k+1}_i:=\frac{1}{N}-\int_{R_i(\lambda^{(k+1)})} \overline{q}^{(k+1)}(x)dx\]
for $i=1,\dots, N$.

\item Define the polytope $P_{k+1}$ as the intersection of $P_k$ and the half-space 

\[\left\{\lambda\in \RR^N: \langle \overline{g}^{(k+1)}, \lambda-\lambda^{(k+1)}\rangle\geq 0\right\}\]  
\end{enumerate}

As $k\rightarrow \infty$ the computed densities $\overline{q}^{k}$ converge to the desired maximizer $q^*(\delta)$. 

\section{Computational examples and applications.}
\label{Sec: Applications}

\subsection{An implementation}

We have implemented the algorithms proposed in this article. They are available as python code, at \url{https://github.com/mauricio-velasco/min-cross-entropy.git}. The algorithm requires solving linear optimization problems and for this it uses the industrial solver GUROBI~\cite{Gurobi}. Full-featured academic licenses for this software (required for running our code) for the purposes of research can be obtained from the vendor. Our implementation can carry out the following tasks (see the file \url{Figures.py}  in the repository for syntax details):
\begin{enumerate}
\item {\bf Compute optimal transports to empirical measures.} Given a probability distribution $p$ (provided as a python function which is able to produce i.i.d. samples from $p$) and an empirical measure $\mu$ (given by the locations of its $N$ data points) computes an optimal transport between $p$ and $\mu$. This optimal transport is encoded by a weight vector $\lambda\in \RR^N$ and is obtained by mapping the weighted voronoi cell $R_i(\lambda)$ to the data point $x_i$. Figure~\ref{Fig: Voronoi} 
shows an optimal transport between the uniform distribution in the square and the empirical measure supported in the black dots. Note that the weighted Voronoi diagram is different from the unweighted voronoi diagram and that in the weighted version the regions have equal probability $1/N$ according to $p$.
\item {\bf Compute the distributions in $B_{\delta}(\mu)$ of minimum cross-entropy with a given $p$.} Given a probability distribution $p$ (specified as a python function which produces i.i.d. samples from $p$), a radius $\delta$ and an empirical measure $\mu$ computes the  distribution in the Wasserstein ball $B_{\delta}(\mu)$ of minimum cross-entropy with $p$. This distribution is specificied by returning $(\lambda^*,\overline{u},\overline{v})$ in the formula for $q_{\delta}^*$ appearing in Theorem~\ref{thm: main}.
Figure~\ref{Fig: MinCrossEntropy} shows the density $q^*_{\delta}$ when $\mu$ is the empirical measure given by the black dots, $p$ is the uniform measure in $[0,1]\times [0,1]$ and various radii $\delta$. 
\end{enumerate} 

\begin{center}
\begin{figure}[h]
    \centering
    \subfloat[][\centering ]{\includegraphics[width=5.2cm]{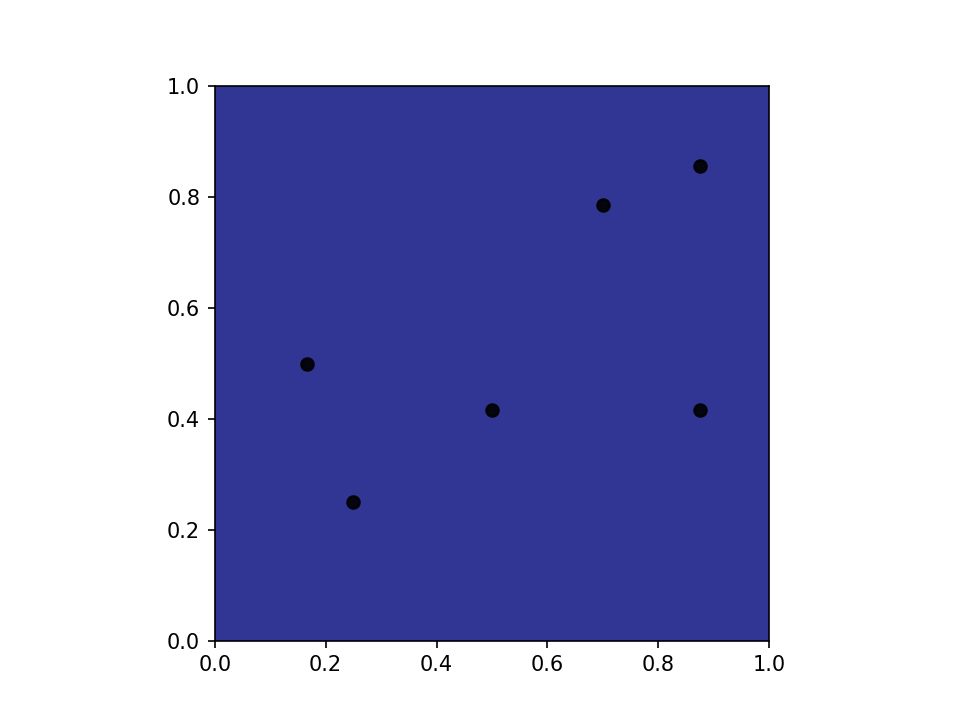}}
    \subfloat[][\centering ]{\includegraphics[width=5.8cm]{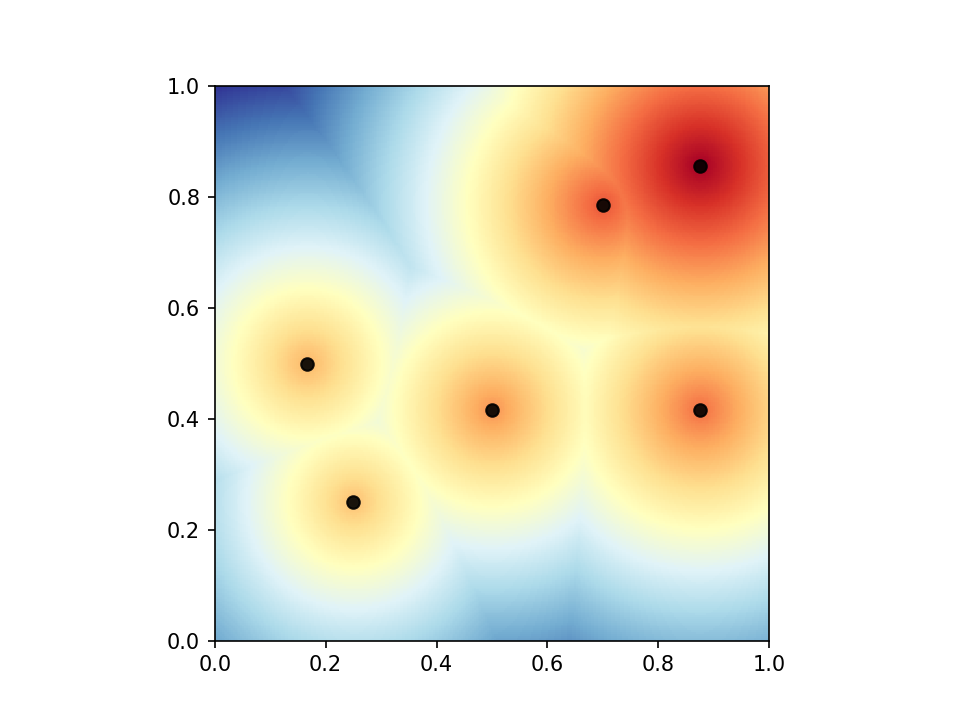}}
    \qquad
    \subfloat[][\centering ]{\includegraphics[width=5.8cm]{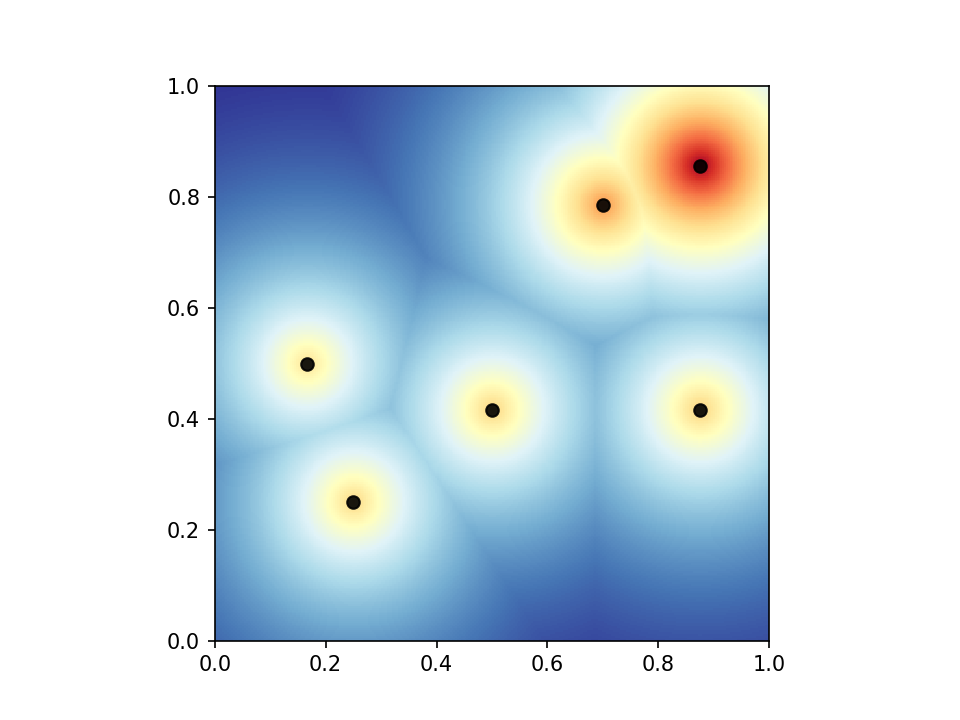}}
    \subfloat[][\centering ]{\includegraphics[width=5.8cm]{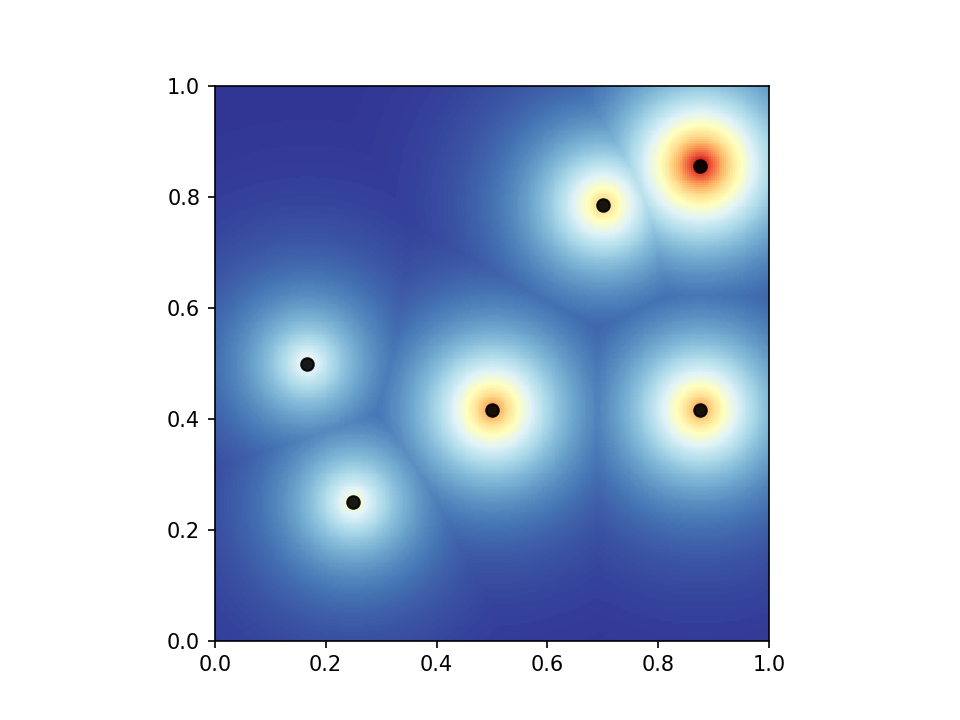}}
    \qquad
    \subfloat[][\centering ]{\includegraphics[width=5.8cm]{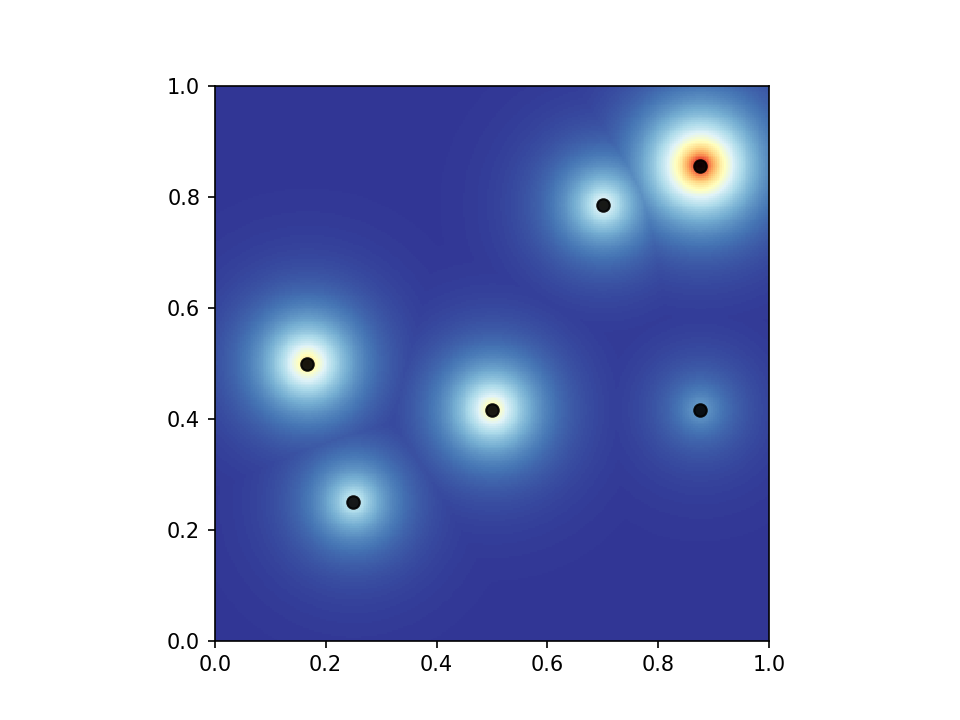}}
    \subfloat[][\centering ]{\includegraphics[width=5.8cm]{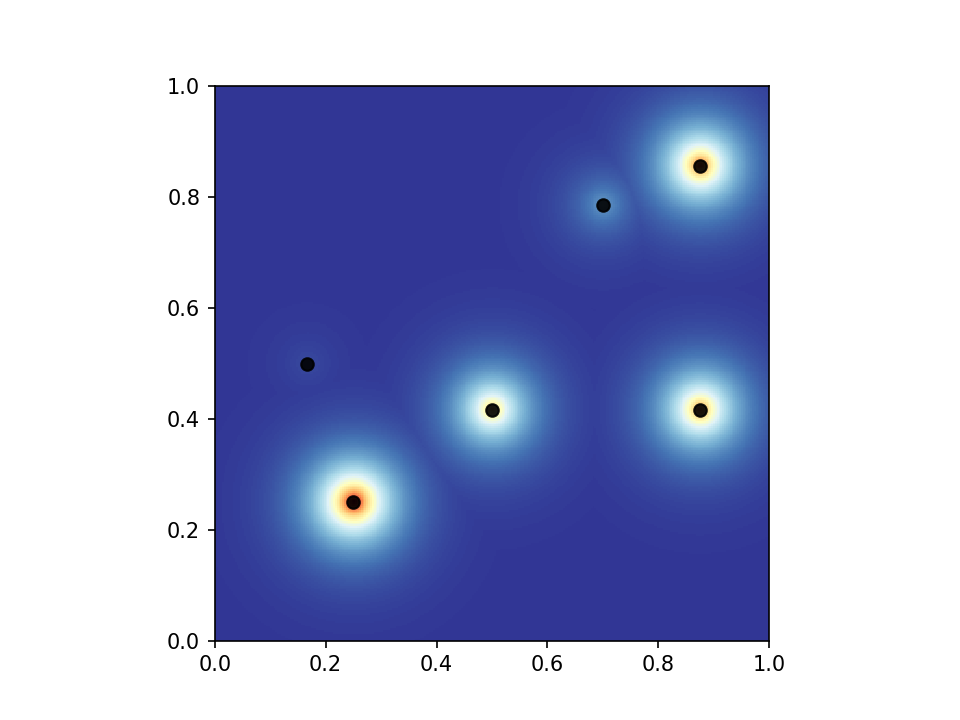}}

\caption{Distributions of minimum cross-entropy $\mathcal{H}(q,p)$ where $p$ is uniform in $[0,1]^2$ and $q\in B_{\delta}(\mu)$ for $\delta\in [0.23, 0.2, 0.15, 0.1, 0.05, 0.02]$.}
    \label{Fig: MinCrossEntropy}%
\end{figure}
\end{center}

\subsection{Bias mitigation via minimum cross-entropy distributions}
The purpose of this Section is to discuss a possible application of our results to the problem of mitigating data-induced biases in machine learning. To make the discussion more concrete we will focus in the case of predictive policing, where this phenomenon is well documented. We begin by briefly reviewing the article~\cite{Significance} which we recommend to the interest reader.

{\bf What is predictive policing?} It is the attempt of using statistical analysis and machine learning algorithms to understand the patterns of criminality to design better policing procedures (more extreme interpretations speak about predicting crime before it happens, in true ''Minority Report" fashion). While this seems like the sort of activity citizens would want their police to be doing, predictive policing software, (which has been built and deployed in many places worldwide within the last few years) and the policing tactics based on it, have raised several serious concerns including  whether the programs unnecessarily target specific (and often disadvantaged) groups more than others. 

More specifically the authors raise in~\cite{Significance} the following rather serious objection: The police data-sets used to train the algorithms are rife with systematic bias. In the authors' words: "Decades of of criminological research, have shown that police records are not a complete census of all criminal offences, nor do they constitute a representative random sample [...].  They measure some complex interaction between criminality, policing strategy, and community-police relations."

To quantitavely assess this bias, the authors study the special case of drug-related crimes in the city of Oakland. They compare the empirical distribution $\mu$ of reported drug arrests and a {\it prior distribution} $p$ constructed by the authors using national health survey data. This prior distribution is a demographically accurate individual-level representation of the real population of the city (in the highest resolution available from the US Census) and estimates the probability of drug use based on sex, household income, age, race, and the geo-coordinates of households using the (well established, well funded, statistically sound) NSDUH survey. The contrast between these two distributions (see~\cite[Figure 1]{Significance}) is rather dramatic and suggests that drug-related police arrests are indeed ratially biased (see~\cite[Figure 1]{Significance}).

Biased training data would affect the outcome of any learning algorithm. However, the relative ease with which fairer or more socially desireable priors can be built (using only publically available data) suggests a possible aproach to limit the effect of these biases. Using the results from this article one may instead learn the distribution:
\[q^*_{\delta} :={\rm argmin}\{H(q,p): q\in B_{\delta}(\mu)\}\]
In words one would like to choose the distribution in $B_{\delta}(\mu)$ which contributes to the prior the smallest possible amount of information. Off course, that still leaves the problem of choosing a good value for the parameter $\delta$. Recall that  $\delta=\infty$ would lead to learning the prior and $\delta=0$ would lead to learning from the police data alone. We believe that this one-parameter choice is an additional {\it desireable feature} of this approach. The choice of $\delta$ should be made by an elected human official (or committeee of experts) who does so taking into account the system of values of her/his society. We also believe that it is an opportunity for a good complementarity between algorithms and humans. It is probably very difficult for a person to estimate the probability distribution of criminality for an entire city while it is much easier to wisely choose one such distribution from a one-parameter family (which the algorithm can probably show to the expert committee in real time with adequate pre-processing by selecting an image as in Figure~\ref{Fig: MinCrossEntropy}.).

\end{document}